\documentclass[11pt,reqno]{amsart}

\usepackage[T1]{fontenc}
\usepackage{lmodern}
\usepackage{amsmath,amssymb,mathtools,mathrsfs}
\usepackage{microtype}
\usepackage[colorlinks=true,linkcolor=blue,citecolor=blue,urlcolor=blue]{hyperref}

\numberwithin{equation}{section}
\allowdisplaybreaks

\newtheorem{theorem}{Theorem}[section]
\newtheorem{proposition}[theorem]{Proposition}
\newtheorem{lemma}[theorem]{Lemma}
\newtheorem{corollary}[theorem]{Corollary}
\theoremstyle{definition}
\newtheorem{conjecture}[theorem]{Conjecture}

\theoremstyle{remark}

\newcommand{\C}{\mathbb C}
\newcommand{\R}{\mathbb R}
\newcommand{\Ric}{\operatorname{Ric}}
\newcommand{\tr}{\operatorname{tr}}

\newcommand{\Euc}{\mathrm{Euc}}

\hypersetup{
  pdftitle={Compact Locally Conformal K\"ahler Manifolds with Constant Chern Holomorphic Sectional Curvature},
  pdfauthor={Zhuzhu Huang and Xueyuan Wan},
  pdfsubject={Compact Locally Conformal K\"ahler Manifolds with Constant Chern Holomorphic Sectional Curvature},
  pdfkeywords={Chern holomorphic sectional curvature, locally conformal Kahler, Bochner-Kahler, complex space form}
}

\title[Compact locally conformal K\"ahler Manifolds ]
{Compact Locally Conformal K\"ahler Manifolds with Constant Chern Holomorphic Sectional Curvature}

\author[Z. Huang]{Zhuzhu Huang}
\address{School of Mathematical Sciences, Chongqing University of Technology, Chongqing 400054, China}
\email{jojo@stu.cqut.edu.cn}

\author[X. Wan]{Xueyuan Wan}
\address{Mathematical Science Research Center, Chongqing University of Technology, Chongqing 400054, China}
\email{xwan@cqut.edu.cn}

\date{}
\makeatletter
\@namedef{subjclassname@2020}{\textup{2020} Mathematics Subject Classification}
\begin{document}

\begin{abstract}
We prove the Chern version of the constant holomorphic sectional curvature
conjecture for compact locally conformal K\"ahler manifolds. More precisely,
let \((M^n,h)\), \(n\geq2\), be a compact locally conformal K\"ahler
manifold whose Chern holomorphic sectional curvature is a constant \(c\). We
show that \(h\) is necessarily K\"ahler and therefore is a complex space form
metric of holomorphic sectional curvature \(c\). In particular, when \(c=0\),
the metric is K\"ahler flat. This removes the nonpositivity assumption from a
theorem of Chen, Chen, and Nie. The proof derives a curvature identity on the
universal K\"ahler cover and shows that the covering metric is
Bochner--K\"ahler. The globally conformally K\"ahler case is then treated by
compact Bochner--K\"ahler rigidity, while the strict LCK case is excluded by
Kamishima's uniformization theorem and the automorphy of the conformal factor.
\end{abstract}

\thanks{Research of Xueyuan Wan is sponsored by the National Key R\&D Program of China (Grant No. 2024YFA1013200).}

\maketitle

\section{Introduction}

Holomorphic sectional curvature is one of the basic curvature invariants in
complex differential geometry. It measures the curvature of complex lines in
the tangent bundle and, in K\"ahler geometry, plays the role of sectional
curvature in Riemannian geometry. The K\"ahler symmetries and polarization
show that the holomorphic sectional curvature determines the full curvature
tensor. Consequently, a complete K\"ahler manifold of constant holomorphic sectional curvature is a complex space form: according to the sign of the
curvature, its simply connected model is complex projective space, complex
Euclidean space, or complex hyperbolic space
\cite{Hawley1953,Igusa1954,KobayashiNomizuII,Zheng2000}.

For a general Hermitian metric, the natural connection preserving both the
metric and the complex structure is the Chern connection. Its curvature need
not satisfy all the K\"ahler symmetries, so the Chern holomorphic sectional
curvature detects only a symmetrized part of the curvature tensor. In
particular, it need not determine the full Chern curvature. Chen, Chen, and
Nie exhibited complete non-K\"ahler Hermitian metrics on \(\C^n\), \(n>2\),
with identically vanishing Chern holomorphic sectional curvature but
nonvanishing Chern curvature \cite{ChenChenNie2021}. Thus compactness is an
essential feature of the following rigidity problem. 
These observations lead to the following conjecture, usually referred to as
the Constant Holomorphic Sectional Curvature Conjecture; see
\cite[Conjecture~1]{Zheng2025Survey}.
\begin{conjecture}[Constant Holomorphic Sectional Curvature Conjecture]
\label{conj:constant-HSC}
 Given any compact Hermitian manifold, if the holomorphic sectional curvature of its Chern (or Levi-Civita) connection is a constant $c$, then when $c \neq 0$ the metric must be K\"ahler (hence it is a complex space form), while when $c=0$ the metric must be Chern (or Levi-Civita) flat.
\end{conjecture}

The two conclusions in Conjecture~\ref{conj:constant-HSC} are genuinely
different. In complex dimension at least three, a compact Chern-flat
Hermitian manifold need not be K\"ahler; by Boothby's theorem, compact
Chern-flat Hermitian manifolds are quotients of complex Lie groups
\cite{Boothby1958}. Hence the zero-curvature case is not merely a limiting
form of the nonzero case.

The Chern version of the conjecture is known in complex dimension two through the work of Balas,
Balas--Gauduchon, and Apostolov--Davidov--Mu\v{s}karov
\cite{Balas1985,BalasGauduchon1985,ApostolovDavidovMuskarov1996}. In higher
dimensions it has been verified under several additional hypotheses,
including for certain twistor spaces \cite{DavidovGrantcharovMuskarov2009},
Chern K\"ahler-like metrics \cite{Tang2021}, complex nilmanifolds
\cite{LiZheng2022}, Bismut K\"ahler-like metrics \cite{RaoZheng2022}, and
broad classes of metrics with parallel Bismut torsion \cite{ChenZhengBTP}.
The general problem remains open; see \cite{Zheng2025Survey} for a recent
survey and further references.

A particularly natural testing ground is the class of locally conformal
K\"ahler manifolds. A Hermitian manifold \((M, h)\), with fundamental form
\(\omega_h\), is locally conformal K\"ahler, or LCK, if
\(d\omega_h=\theta\wedge\omega_h\) for a closed real one-form \(\theta\),
called the Lee form. The metric is globally conformally K\"ahler precisely
when \(\theta\) is exact \cite{DragomirOrnea1998,Vaisman1980}. On the
universal covering \(\pi:\widetilde M\to M\), there are a K\"ahler metric
\(\widetilde h\) and a smooth real-valued function \(f\) such that
\[
\pi^*h=e^f\widetilde h,
\qquad
 df=\pi^*\theta.
\]
Every deck transformation acts on \((\widetilde M,\widetilde h)\) by a
holomorphic homothety. Thus the failure of global conformal K\"ahlerness is
encoded by a homothety character of the deck group. From this perspective,
the constant holomorphic sectional curvature problem asks whether local
curvature rigidity is strong enough to eliminate this global monodromy.

Chen, Chen, and Nie made an important advance in this direction. They proved
that a compact LCK metric with constant nonpositive Chern holomorphic
sectional curvature is K\"ahler \cite[Theorem~1.1]{ChenChenNie2021}. More
generally, they showed that an LCK metric with pointwise constant
nonpositive Chern holomorphic sectional curvature is globally conformally
K\"ahler \cite[Theorem~1.2]{ChenChenNie2021}. Their argument uses relations
among the Chern--Ricci tensors and scalar curvatures, followed by global
conformal and integral identities. The sign condition enters essentially in
that method, leaving the positive constant case open for general compact LCK
manifolds. The positive case was subsequently obtained for the more
restrictive class of Vaisman metrics through results on Bismut-torsion
parallel geometry \cite{ChenZhengBTP}.

The purpose of the present paper is to remove the sign restriction entirely.
Our main result is the following.

\begin{theorem}\label{thm:main}
Let \((M^n,h)\) be a connected compact LCK manifold of complex dimension
\(n\geq2\). If its Chern holomorphic sectional curvature satisfies
\(H_h^C\equiv c\) for some \(c\in\R\), then \(h\) is K\"ahler. Moreover,
\[
R^h_{i\bar j k\bar\ell}
=
\frac{c}{2}
\bigl(
 h_{i\bar j}h_{k\bar\ell}
 +h_{i\bar\ell}h_{k\bar j}
\bigr).
\]
Consequently, \(h\) is a complex space form metric of holomorphic sectional
curvature \(c\). In particular, if \(c=0\), then \(h\) is K\"ahler flat.
\end{theorem}

Theorem~\ref{thm:main} proves the Chern version of
Conjecture~\ref{conj:constant-HSC} for all compact LCK manifolds and removes
the hypothesis \(c\leq0\) from \cite[Theorem~1.1]{ChenChenNie2021}. In the
zero-curvature case the conclusion is stronger than the one predicted for a
general Hermitian manifold: the non-K\"ahler Chern-flat alternative cannot
occur in the LCK category.

The proof is driven by a conformal curvature identity on the universal
K\"ahler cover. With the notation above, the curvature tensor of
\(\widetilde h\) is shown to satisfy
\[
R^{\widetilde h}
=
\frac{ce^f}{2}G_{\widetilde h}
+
\frac14L_{\widetilde h}(A),
\qquad
A_{i\bar j}=f_{i\bar j},
\]
where \(G_{\widetilde h}\) is the complex-space-form curvature tensor and
\(L_{\widetilde h}\) is the standard \(U(n)\)-equivariant operator on
\((1,1)\)-tensors. This identity has no component in the Bochner summand of
the K\"ahler curvature decomposition. It follows that \(\widetilde h\) is
Bochner--K\"ahler, or equivalently that \(h\) is Bochner-flat in the
conformally invariant sense of Tricerri--Vanhecke
\cite{TricerriVanhecke1981,Kamishima2006}.

We then distinguish two branches. When the Lee form is exact,
\(\widetilde h\) descends to a compact Bochner--K\"ahler metric in the
conformal class of \(h\). 
The compact classification of Bochner--K\"ahler manifolds, due to Kamishima
and also recovered in Bryant's work, implies that every connected compact
smooth Bochner--K\"ahler manifold is a compact quotient of a symmetric
Bochner--K\"ahler model; in particular, this metric is locally symmetric
\cite{Kamishima1994,Kamishima2005Correction,Bryant2001}.
 Differentiating the curvature
identity and using K\"ahler commutation relations produces an
overdetermined equation for the complex Hessian of the conformal factor,
which forces that factor to be constant. When the Lee form is not exact,
Kamishima's uniformization theorem, together with Fried's classification of
closed similarity manifolds, identifies the universal K\"ahler cover with
punctured Euclidean space and supplies a strict similarity in the deck group
\cite{Kamishima2006,Fried1980}. The flat form of the curvature identity is
then incompatible with the automorphy of \(f\), both when \(c\neq0\) and
when \(c=0\).

The paper is organized as follows. Section~2 reviews Chern holomorphic
sectional curvature and the universal K\"ahler cover of an LCK manifold. In
Section~3 we derive the curvature identity and prove Bochner-flatness.
Sections~4 and~5 treat the globally conformally K\"ahler and strict LCK
branches, respectively. The proof of Theorem~\ref{thm:main} is completed in
Section~6.

\

{\bf Acknowledgements:} The authors are grateful to Professors Fangyang Zheng and Xiaolan Nie for their interest in this work, encouraging comments, and helpful communications. The
second author is supported by the National Key R\&D Program of China
(Grant No. 2024YFA1013200).

\vspace{1.5mm}

\section{Preliminaries}

Throughout the paper, manifolds are connected and have complex dimension
\(n\geq2\). We use the Einstein summation convention. All curvature
components are written in a local holomorphic frame and follow the convention
specified below.

\subsection{Chern holomorphic sectional curvature}

Let \(M^n\) be a connected complex manifold of
complex dimension \(n\). Let \((M^n,h)\) be a Hermitian manifold. For a
local holomorphic coordinate system \((z^1,\ldots,z^n)\), we write
\[
        h_{i\bar j}
        =
        h\left(
        \frac{\partial}{\partial z^i},
        \frac{\partial}{\partial \bar z^j}
        \right),
\]
and denote by
\[
        \omega_h
        =
        \sqrt{-1}\,h_{i\bar j}\,dz^i\wedge d\bar z^j
\]
the fundamental form of \(h\).

The Chern connection \(\nabla^C\) is the unique Hermitian connection on
\(T^{1,0}M\) whose \((0,1)\)-part is \(\bar\partial\). In local holomorphic
coordinates, the Chern connection $\nabla^C$ is given by
\[
\nabla^C_{\frac{\partial}{\partial z^i}}\frac{\partial}{\partial z^j}=\Gamma^k_{ij}\frac{\partial}{\partial z^k},\quad \Gamma^k_{ij}=h^{k\bar\ell}\partial_i h_{j\bar\ell},
\] and we use the curvature convention
\[
R^C_{i\bar j k\bar\ell}
=
-h_{p\bar\ell}\,\partial_{\bar j}\Gamma^p_{ik}
=
-\partial_i\partial_{\bar j}h_{k\bar\ell}
+h^{p\bar q}\partial_i h_{k\bar q}\partial_{\bar j}h_{p\bar\ell}.
\]
For \(0\neq X=X^i\partial_i\in T^{1,0}_pM\), its Chern holomorphic sectional
curvature is
\begin{equation}\label{eq:Chern-HSC}
H_h^C(X)
=
\frac{R^C_{i\bar j k\bar\ell}X^i\bar X^jX^k\bar X^\ell}
     {(h_{i\bar j}X^i\bar X^j)^2}.
\end{equation}
This depends only on the complex line spanned by \(X\).

For a non-K\"ahler metric, the Chern curvature tensor need not have the
K\"ahler symmetries. The quantity in \eqref{eq:Chern-HSC} therefore detects
the symmetrized tensor
\begin{equation}\label{eq:Chern-symmetrization}
K^C_{i\bar j k\bar\ell}
=
\frac14\bigl(
 R^C_{i\bar j k\bar\ell}
 +R^C_{k\bar j i\bar\ell}
 +R^C_{i\bar\ell k\bar j}
 +R^C_{k\bar\ell i\bar j}
\bigr).
\end{equation}
The following polarization statement is standard; see
\cite{Balas1985} and \cite[Proposition~3.4]{ChenChenNie2021}.

\begin{lemma}\label{lem:Chern-polarization}
At a point \(p\in M\), one has \[H_h^C(X)=c\] for every
\(0\neq X\in T^{1,0}_pM\) if and only if
\begin{equation}\label{eq:Chern-polarization}
K^C_{i\bar j k\bar\ell}
=
\frac{c}{2}
\bigl(
 h_{i\bar j}h_{k\bar\ell}
 +h_{i\bar\ell}h_{k\bar j}
\bigr)
\end{equation}
at \(p\).
\end{lemma}

\subsection{LCK metrics and their universal K\"ahler covers}

A Hermitian metric \(h\) is locally conformal K\"ahler if its fundamental
form satisfies
\begin{equation}\label{eq:LCK}
d\omega_h=\theta\wedge\omega_h,
\qquad
d\theta=0,
\end{equation}
for a real one-form \(\theta\), called the Lee form. Equivalently, every
point has a neighborhood on which a conformal multiple of \(h\) is K\"ahler.
The metric is globally conformally K\"ahler if \(\theta\) is exact; see
\cite{DragomirOrnea1998,Vaisman1980}.

Let \(\pi:\widetilde M\to M\) be the universal covering. Since
\(\pi^*\theta\) is exact, choose \(f\in C^\infty(\widetilde M,\R)\) such that
\(df=\pi^*\theta\). Then
\begin{equation}\label{eq:universal-Kahler-cover}
\widetilde h=e^{-f}\pi^*h
\qquad\text{and hence}\qquad
\pi^*h=e^f\widetilde h
\end{equation}
defines a K\"ahler metric on \(\widetilde M\). We refer to
\((\widetilde M,\widetilde h)\) as the universal K\"ahler cover associated
with \(h\). Changing \(f\) by a constant only rescales \(\widetilde h\).

Let \(\gamma\) be a deck transformation. Since \(\pi\circ\gamma=\pi\), the
metric \(\pi^*h=e^f\widetilde h\) is deck invariant, and therefore
\[
e^{f\circ\gamma}\gamma^*\widetilde h=e^f\widetilde h.
\]
Moreover,
\(
d(f\circ\gamma)=\gamma^*df=(\pi\circ\gamma)^*\theta=df
\), so \(f\circ\gamma-f\) is constant. Thus there is a uniquely determined
number \(\lambda_\gamma>0\) such that
\begin{equation}\label{eq:deck-automorphy}
f\circ\gamma=f-2\log\lambda_\gamma,
\qquad
\gamma^*\widetilde h=\lambda_\gamma^2\widetilde h.
\end{equation}
Hence every deck transformation is a holomorphic homothety of
\((\widetilde M,\widetilde h)\), and
\(\gamma\mapsto\lambda_\gamma\) is a multiplicative character of the deck
group.

The Lee form is exact if and only if this character is trivial. Indeed, if
\(\theta=du\) on \(M\), then, after adding a constant to \(u\), one has
\(f=\pi^*u\), so \(f\) is deck invariant and \(\lambda_\gamma=1\) for every
\(\gamma\). Conversely, if every \(\lambda_\gamma=1\), then \(f\) descends
to a function \(u\) on \(M\), and \(df=\pi^*du=\pi^*\theta\) implies
\(du=\theta\). Equivalently, the Lee form is exact precisely when all deck
transformations act isometrically on \((\widetilde M,\widetilde h)\).

\section{The curvature identity and Bochner-flatness}

We now pass to the universal K\"ahler cover
\(\pi:\widetilde M\to M\), where \(\pi^*h=e^f\widetilde h\). The goal of
this section is to translate the condition \(H_h^C\equiv c\) into a
curvature identity for the K\"ahler metric \(\widetilde h\). This identity
shows that the Bochner component of the curvature of \(\widetilde h\)
vanishes, and hence that the original LCK metric \(h\) is Bochner-flat in
the sense of Tricerri--Vanhecke \cite{TricerriVanhecke1981}; see also
\cite[\S7]{Kamishima2006} for the role of this conformal Bochner tensor in
Bochner-flat LCK geometry. This reduction will be used in both the
globally conformally K\"ahler and the strict LCK branches.

\subsection{The conformal curvature identity}

We first introduce two algebraic tensors. Let \((V,q)\) be a Hermitian
vector space and let \(B\) be a complex \((1,1)\)-tensor. Define
\begin{align}
(G_q)_{i\bar j k\bar\ell}
&=
q_{i\bar j}q_{k\bar\ell}
+q_{i\bar\ell}q_{k\bar j},
\label{eq:G}
\\
(L_q(B))_{i\bar j k\bar\ell}
&=
B_{i\bar j}q_{k\bar\ell}
+B_{k\bar j}q_{i\bar\ell}
+B_{i\bar\ell}q_{k\bar j}
+B_{k\bar\ell}q_{i\bar j}.
\label{eq:Loperator}
\end{align}
Thus \(\frac{c}{2}G_q\) has holomorphic sectional curvature \(c\), and
\(L_q(q)=2G_q\).

We shall use the following elementary observation twice.

\begin{lemma}\label{lem:L-linear-algebra}
Let \((V,q)\) have complex dimension \(n\geq2\). If
\(L_q(B)=\mu G_q\) for a complex \((1,1)\)-tensor \(B\) and
\(\mu\in\C\), then \(B=\frac{\mu}{2}q\). In particular, \(L_q\) is
injective.
\end{lemma}

\begin{proof}
Contracting the last two indices gives
\begin{equation}\label{eq:contract-L}
(n+2)B_{i\bar j}+(\tr_qB)q_{i\bar j}
=
\mu(n+1)q_{i\bar j}.
\end{equation}
Taking the \(q\)-trace yields \(\tr_qB=\frac{\mu n}{2}\). Substitution into
\eqref{eq:contract-L} gives \(B=\frac{\mu}{2}q\).
\end{proof}

Return now to the universal K\"ahler cover
\(\pi:(\widetilde M,\widetilde h)\to(M,h)\), and put
\(\widehat h=\pi^*h=e^f\widetilde h\). Let \(\widehat R^C\) be the Chern
curvature of \(\widehat h\), and let \(R^{\widetilde h}\) be the K\"ahler
curvature of \(\widetilde h\). The conformal change formula for Chern
curvature is
\begin{equation}\label{eq:conformal-curvature}
\widehat R^C_{i\bar j k\bar\ell}
=
e^f\bigl(
 R^{\widetilde h}_{i\bar j k\bar\ell}
 -f_{i\bar j}\widetilde h_{k\bar\ell}
\bigr),
\end{equation}
where \(f_{i\bar j}=\partial_i\partial_{\bar j}f\); compare
\cite[(3.17)]{ChenChenNie2021}.

\begin{proposition}\label{prop:curvature-equation}
Assume that \(H_h^C\equiv c\). On the universal K\"ahler cover one has
\begin{equation}\label{eq:main-curv-eq}
R^{\widetilde h}
=
\frac{\alpha}{2}G_{\widetilde h}
+
\frac14L_{\widetilde h}(A),
\qquad
\alpha=ce^f,
\qquad
A_{i\bar j}=f_{i\bar j}.
\end{equation}
\end{proposition}

\begin{proof}
Symmetrizing \eqref{eq:conformal-curvature} and using the K\"ahler
symmetries of \(R^{\widetilde h}\), we obtain
\begin{equation}\label{eq:symmetrized-conformal-curvature}
\widehat K^C
=
e^f\left(
 R^{\widetilde h}-\frac14L_{\widetilde h}(A)
\right).
\end{equation}
Since \(\widehat h=\pi^*h\) also has constant Chern holomorphic sectional
curvature \(c\), Lemma~\ref{lem:Chern-polarization} gives
\[
\widehat K^C
=
\frac{c}{2}G_{\widehat h}
=
\frac{ce^{2f}}{2}G_{\widetilde h}.
\]
Comparing this with \eqref{eq:symmetrized-conformal-curvature} and dividing
by \(e^f\) yields \eqref{eq:main-curv-eq}.
\end{proof}

\subsection{The K\"ahler curvature decomposition}

Let \((V,q)\) be a Hermitian vector space of complex dimension
\(n\geq2\), and let \(\mathcal K(V)\) denote the real vector space of
algebraic K\"ahler curvature tensors on \(V\). Thus, in a unitary basis,
an element \(R\in\mathcal K(V)\) satisfies
\[
R_{i\bar j k\bar\ell}
=
R_{k\bar j i\bar\ell}
=
R_{i\bar\ell k\bar j},
\qquad
\overline{R_{i\bar j k\bar\ell}}
=
R_{j\bar i\ell\bar k}.
\]
Its Ricci and scalar contractions are defined by
\[
\Ric_{i\bar j}
=
q^{k\bar\ell}R_{i\bar j k\bar\ell},
\qquad
s=q^{i\bar j}\Ric_{i\bar j},
\]
and the trace-free Ricci tensor is
\(\Ric^0=\Ric-\frac{s}{n}q\).

Let \(\mathcal H(V)\) denote the real vector space of Hermitian
\((1,1)\)-tensors on \(V\), and let
\[
\mathcal H_0(V)
=
\{S\in\mathcal H(V):\operatorname{tr}_qS=0\}.
\]
For \(S\in\mathcal H(V)\), recall that
\[
\begin{aligned}
(G_q)_{i\bar j k\bar\ell}
&=
q_{i\bar j}q_{k\bar\ell}
+
q_{i\bar\ell}q_{k\bar j},\\
(L_q(S))_{i\bar j k\bar\ell}
&=
S_{i\bar j}q_{k\bar\ell}
+
S_{k\bar j}q_{i\bar\ell}
+
S_{i\bar\ell}q_{k\bar j}
+
S_{k\bar\ell}q_{i\bar j}.
\end{aligned}
\]
Contracting in the last two indices gives
\[
q^{k\bar\ell}(G_q)_{i\bar j k\bar\ell}
=
(n+1)q_{i\bar j}
\]
and
\[
q^{k\bar\ell}(L_q(S))_{i\bar j k\bar\ell}
=
(n+2)S_{i\bar j}
+
(\operatorname{tr}_qS)q_{i\bar j}.
\]
In particular, \(L_q(q)=2G_q\), while for
\(S\in\mathcal H_0(V)\),
\[
q^{k\bar\ell}(L_q(S))_{i\bar j k\bar\ell}
=
(n+2)S_{i\bar j}.
\]

With respect to the natural inner product induced by \(q\), the space
\(\mathcal K(V)\) admits the orthogonal \(U(n)\)-irreducible
decomposition
\begin{equation}\label{eq:Kahler-curvature-decomposition-space}
\mathcal K(V)
=
\mathcal B(V)
\oplus
L_q\bigl(\mathcal H_0(V)\bigr)
\oplus
\mathbb R G_q,
\end{equation}
where
\[
\mathcal B(V)
=
\left\{
R\in\mathcal K(V):
q^{k\bar\ell}R_{i\bar j k\bar\ell}=0
\right\}
\]
is the Bochner curvature module. Equivalently, every
\(R\in\mathcal K(V)\) has the unique decomposition
\begin{equation}\label{eq:Kahler-curvature-decomposition}
R
=
B_q(R)
+
\frac{1}{n+2}L_q(\Ric^0)
+
\frac{s}{n(n+1)}G_q.
\end{equation}
The coefficients in \eqref{eq:Kahler-curvature-decomposition} follow
directly from the preceding contraction identities and correspond to
the convention \(s=q^{i\bar j}\Ric_{i\bar j}\). The tensor \(B_q(R)\)
has vanishing Ricci contraction and is called the Bochner curvature
tensor of \(R\).

This decomposition goes back to Bochner \cite{Bochner1949}; see also
\cite[(2.63)]{Besse1987} and
\cite[\S2.1.3]{Bryant2001}. Notice that the normalization of the scalar
curvature varies in the literature, whereas
\eqref{eq:Kahler-curvature-decomposition} is normalized according to
the contraction conventions used here.

\medskip

\noindent\textbf{Bochner-flatness terminology.}
A K\"ahler metric \(k\) is called \emph{Bochner--K\"ahler}, or
\emph{Bochner-flat}, if its Bochner curvature tensor vanishes
identically:
\[
B(k)\equiv0.
\]
For a general Hermitian metric \(k\), let \(B^{\mathrm{TV}}(k)\) denote
the conformally invariant Bochner component in the Tricerri--Vanhecke
\(U(n)\)-decomposition of the Riemannian curvature tensor. The metric \(k\) is said to be
\emph{Bochner-flat in the sense of Tricerri--Vanhecke} if
\(B^{\mathrm{TV}}(k)\equiv0\). Its vanishing
is preserved under Hermitian conformal changes, and on K\"ahler metrics
it agrees with the ordinary K\"ahler Bochner tensor. In complex dimension
two, we use the corrected four-dimensional formula; see
\cite{TricerriVanhecke1981,EuhEtAl2011,BandeBlairHadjar2015}.

\begin{corollary}\label{cor:Bochner-flat-cover}
The universal K\"ahler cover
\((\widetilde M,\widetilde h)\) is Bochner--K\"ahler. Equivalently, the
original LCK metric \(h\) is Bochner-flat in the sense of
Tricerri--Vanhecke.
\end{corollary}

\begin{proof}
By \eqref{eq:main-curv-eq}, the curvature tensor of
\(\widetilde h\) satisfies
\[
R^{\widetilde h}
=
\frac{\alpha}{2}G_{\widetilde h}
+
\frac14L_{\widetilde h}(A),
\qquad
\alpha=ce^f,
\qquad
A_{i\bar j}=f_{i\bar j}.
\]
Decompose \(A\) into its trace-free and pure-trace parts:
\[
A
=
A^0
+
\frac{\operatorname{tr}_{\widetilde h}A}{n}\widetilde h.
\]
Since
\(L_{\widetilde h}(\widetilde h)=2G_{\widetilde h}\), we obtain
\[
R^{\widetilde h}
=
\frac14L_{\widetilde h}(A^0)
+
\left(
\frac{\alpha}{2}
+
\frac{\operatorname{tr}_{\widetilde h}A}{2n}
\right)G_{\widetilde h}.
\]
Thus, pointwise,
\[
R^{\widetilde h}
\in
L_{\widetilde h}\bigl(\mathcal H_0(T^{1,0}\widetilde M)\bigr)
\oplus
\mathbb R G_{\widetilde h}.
\]
It follows from
\eqref{eq:Kahler-curvature-decomposition-space} that the Bochner
component of \(R^{\widetilde h}\) vanishes. Hence
\((\widetilde M,\widetilde h)\) is Bochner--K\"ahler.

It remains to relate this conclusion to the original metric \(h\).
The Tricerri--Vanhecke tensor is natural under pullback by local
diffeomorphisms, so
\[
B^{\mathrm{TV}}(\pi^*h)
=
\pi^*B^{\mathrm{TV}}(h).
\]
Since \(\pi\) is surjective, this gives
\[
B^{\mathrm{TV}}(h)\equiv0
\quad\Longleftrightarrow\quad
B^{\mathrm{TV}}(\pi^*h)\equiv0.
\]
Using \(\pi^*h=e^f\widetilde h\), the conformal invariance of the
vanishing condition gives
\[
B^{\mathrm{TV}}(\pi^*h)\equiv0
\quad\Longleftrightarrow\quad
B^{\mathrm{TV}}(\widetilde h)\equiv0.
\]
Finally, since \(\widetilde h\) is K\"ahler,
\[
B^{\mathrm{TV}}(\widetilde h)\equiv0
\quad\Longleftrightarrow\quad
B(\widetilde h)\equiv0.
\]
Combining these equivalences proves the result.
\end{proof}
\section{The globally conformally K\"ahler branch}

We first treat the case in which the Lee form is exact.

\begin{proposition}\label{prop:GCK-branch}
Let \((M^n,h)\) be a connected compact globally conformally K\"ahler
Hermitian manifold of complex dimension \(n\geq2\). If \(H_h^C\equiv c\),
then \(h\) is K\"ahler.
\end{proposition}

\begin{proof}
Let \(\theta\) be the Lee form. Since \(\theta\) is exact, write
\(\theta=du\) for some \(u\in C^\infty(M,\R)\). After adding a constant to
\(u\), the identities \(df=\pi^*\theta\) and
\(\pi^*h=e^f\widetilde h\) give
\begin{equation}\label{eq:GCK-descent}
f=\pi^*u,
\qquad
h_0=e^{-u}h,
\qquad
\widetilde h=\pi^*h_0.
\end{equation}
In particular, \(h_0\) is a K\"ahler metric on the compact manifold \(M\).

Since \(\pi^*h_0=\widetilde h\) and Bochner-flatness is a local
curvature condition, Corollary~\ref{cor:Bochner-flat-cover} implies that
\(h_0\) is Bochner--K\"ahler. The compact classification of Bochner--K\"ahler manifolds, due to Kamishima
and also recovered by Bryant, implies that every connected compact smooth
Bochner--K\"ahler manifold is a compact quotient of a symmetric
Bochner--K\"ahler model \(M_\kappa^p\times M_{-\kappa}^{n-p}\);
see \cite{Kamishima1994,Kamishima2005Correction} and
\cite[Corollary~4.17]{Bryant2001}. Thus the metric is locally symmetric.
Here the word ``smooth'' is important: in the orbifold category there are
additional compact Bochner--K\"ahler examples, such as weighted projective
spaces, which are not covered by the smooth classification; see
\cite[Remark~4.18 and Theorem~4.29]{Bryant2001}. In Bryant's proof, compactness forces the cohomogeneity \(m\) to be zero.
Thus the momentum mapping is constant and the metric is locally
homogeneous. Since scalar curvature is invariant under local isometries, it
is constant. Proposition~2.5 of \cite{Bryant2001} then implies that the
metric is locally symmetric, and hence
\begin{equation}\label{eq:nablaR-zero}
    \nabla R=0.
\end{equation}

By \eqref{eq:GCK-descent}, the curvature identity
\eqref{eq:main-curv-eq} descends to \(M\) as
\begin{equation}\label{eq:main-curv-eq-GCK}
R
=
\frac{\alpha}{2}G_{h_0}
+
\frac14L_{h_0}(A),
\qquad
\alpha=ce^u,
\qquad
A_{i\bar j}=\nabla_i\nabla_{\bar j}u.
\end{equation}
Differentiating in a holomorphic direction and using
\eqref{eq:nablaR-zero} together with \(\nabla h_0=0\), we obtain
\begin{equation}\label{eq:diff-main-eq}
L_{h_0}(\nabla_mA)=-2\alpha_mG_{h_0},
\qquad
\alpha_m=\nabla_m\alpha.
\end{equation}
Lemma~\ref{lem:L-linear-algebra} therefore yields
\begin{equation}\label{eq:nablaA-alpha}
\nabla_mA_{i\bar j}
=-\alpha_m(h_0)_{i\bar j}.
\end{equation}

Because the curvature of a K\"ahler connection is of type \((1,1)\), the
holomorphic covariant derivatives of \(\bar\partial u\) commute; see
\cite[Chapter~IX, \S4]{KobayashiNomizuII}. Hence
\(
\nabla_mA_{i\bar j}=\nabla_iA_{m\bar j}
\).  Combining this with \eqref{eq:nablaA-alpha}, we find
\begin{equation}\label{eq:alpha-rank-one}
\alpha_m(h_0)_{i\bar j}
=
\alpha_i(h_0)_{m\bar j}.
\end{equation}
At a given point, choose an \(h_0\)-unitary frame. For each fixed \(m\),
choose \(i\neq m\) and set \(j=i\) in
\eqref{eq:alpha-rank-one}. This gives \(\alpha_m=0\). Since \(n\geq2\),
we conclude that \(\partial\alpha=0\).

Suppose first that \(c\neq0\). Since \(\alpha=ce^u\), it follows that
\(\partial u=0\). The function \(u\) is real-valued, so it is constant.

It remains to consider \(c=0\). Then \(\alpha=0\), and
\eqref{eq:nablaA-alpha} gives \(\nabla^{1,0}A=0\). Since \(u\) is
real-valued, \(A\) is Hermitian; taking complex conjugates gives
\(\nabla^{0,1}A=0\), and hence \(\nabla A=0\). Consequently, the real
\((1,1)\)-form
\[
\eta=\sqrt{-1}\,\partial\bar\partial u
\]
is parallel. It is also exact: with
\(
d^cu=\frac{\sqrt{-1}}2(\bar\partial u-\partial u)
\), one has \(\eta=dd^cu\). Since \(\eta\) is parallel, it is coclosed. Indeed, for any local
\(h_0\)-orthonormal frame \(\{e_a\}\), the codifferential is given by
\[
        d^*\eta
        =
        -\sum_a e_a\lrcorner\,\nabla_{e_a}\eta ;
\]
hence \(\nabla\eta=0\) implies \(d^*\eta=0\). See, for instance,
\cite[Chapter~3]{Jost2017}. So compactness and
integration by parts give
\[
\int_M|\eta|_{h_0}^2\,dV_{h_0}
=
\int_M\langle dd^cu,\eta\rangle_{h_0}\,dV_{h_0}
=
\int_M\langle d^cu,d^*\eta\rangle_{h_0}\,dV_{h_0}
=0.
\]
Thus \(\eta=0\). Taking its \(h_0\)-trace shows that \(u\) is harmonic, so
\(u\) is constant on the compact connected manifold \(M\).

In both cases \(h=e^u h_0\) is a constant multiple of the K\"ahler metric
\(h_0\), and is therefore K\"ahler.
\end{proof}

\section{The strict LCK branch}

It remains to exclude the case in which the Lee form is not exact. We use
the following compact consequence of Kamishima's uniformization theorem
and Fried's classification of closed similarity manifolds.

\begin{proposition}[Kamishima--Fried uniformization]\label{prop:Kamishima}
Let \((M^n,h)\) be a connected compact LCK manifold of complex dimension
\(n\geq 2\). Suppose that the Lee form of \(h\) is not exact and that the
Tricerri--Vanhecke Bochner tensor of \(h\) vanishes. After replacing the
K\"ahler metric \(\widetilde h\) on the universal cover by a positive constant
multiple, there exists a holomorphic isometry
\[
        \Phi:(\widetilde M,\widetilde h)
        \longrightarrow
        \bigl(\mathbb C^n\setminus\{0\},h_{\mathrm{Euc}}\bigr),
\]
where \(h_{\mathrm{Euc}}\) is the standard flat K\"ahler metric. Under this
identification, every deck transformation is of the form
\[
        \gamma(z)=r_\gamma U_\gamma z,
        \qquad
        r_\gamma>0,\quad U_\gamma\in U(n).
\]
Moreover, \(r_\gamma\neq 1\) for at least one deck transformation \(\gamma\).
\end{proposition}

\begin{proof}
Let \(\theta\) be the Lee form of \(h\), and let
\(\pi:\widetilde M\to M\) be the universal covering. Choose
\(f\in C^\infty(\widetilde M,\mathbb R)\) such that \(df=\pi^*\theta\), and
write
\[
        \pi^*h=e^f\widetilde h .
\]
Then \(\widetilde h=e^{-f}\pi^*h\) is the canonical K\"ahler metric on the
universal cover. The function \(f\) is determined only up to an additive
constant; correspondingly, \(\widetilde h\) is determined only up to
multiplication by a positive constant.

Since the Tricerri--Vanhecke Bochner tensor is conformally invariant and
agrees with the ordinary Bochner tensor on K\"ahler metrics, the assumption
\(B^{\mathrm{TV}}(h)=0\) implies that \((\widetilde M,\widetilde h)\) is
Bochner--K\"ahler. Because \(\theta\) is not exact, the LCK metric \(h\) is
not globally conformally K\"ahler. By Vaisman's theorem, a compact LCK
manifold is globally conformally K\"ahler if and only if its underlying
complex manifold admits a K\"ahler metric; see \cite[Theorem~2.1]{Vaisman1980}.
Thus the underlying complex manifold \(M\) is non-K\"ahlerian in the sense
used by Kamishima.

Kamishima's uniformization theorem for non-K\"ahler Bochner-flat LCK
manifolds \cite[Theorem~7.1]{Kamishima2006} gives an equivariant
holomorphic isometric immersion of the canonical K\"ahler cover into one of
the two model geometries denoted by \((i)'\) and \((iii)'\) in
\cite{Kamishima2006}. Since \(M\) is compact, the model \((iii)'\) is
excluded by \cite[Proposition~3.8]{Kamishima2006}. Hence only the Euclidean
similarity model \((i)'\) can occur. After choosing the multiplicative
normalization of \(\widetilde h\), Kamishima's developing map satisfies
\[
        D^*h_{\mathrm{Euc}}=\widetilde h
\]
and is equivariant with respect to a holonomy homomorphism
\[
        \rho:\mathrm{Deck}(\pi)\longrightarrow
        \mathbb C^n\rtimes\bigl(U(n)\times\mathbb R^+\bigr).
\]
Here \(\mathrm{Deck}(\pi)\) denotes the deck transformation group of the covering
\(\pi:\widetilde M\to M\), namely the group of all diffeomorphisms
\(\gamma:\widetilde M\to\widetilde M\) satisfying \(\pi\circ\gamma=\pi\).

Thus
\[
        D\circ\gamma=\rho(\gamma)\circ D,
        \qquad
        \gamma\in\mathrm{Deck}(\pi).
\]
Writing \(\rho(\gamma)\) explicitly, we have
\[
        \rho(\gamma)(z)=a_\gamma+r_\gamma U_\gamma z,
        \qquad
        a_\gamma\in\mathbb C^n,\quad r_\gamma>0,\quad U_\gamma\in U(n).
\]

We now pass from the complex similarity structure to a real one. Identifying
\(\mathbb C^n\) with \(\mathbb R^{2n}\), every unitary transformation is
orthogonal for the underlying Euclidean metric. Hence
\[
        \mathbb C^n\rtimes\bigl(U(n)\times\mathbb R^+\bigr)
        \subset
        \mathbb R^{2n}\rtimes\bigl(O(2n)\times\mathbb R^+\bigr).
\]
Therefore \(D\), together with the above holonomy representation, defines a
real affine similarity structure on the closed manifold \(M\) in the sense of
Fried.

By \cite[Corollary~7.2]{Kamishima2006}, the canonical K\"ahler metric
\(\widetilde h\) is incomplete. Since \(D^*h_{\mathrm{Euc}}=\widetilde h\),
the induced real similarity structure is incomplete. Fried's theorem then
applies. By \cite[Theorem~1]{Fried1980}, an incomplete closed similarity
manifold is radiant; that is, its similarity holonomy fixes a common point
\(p\in\mathbb R^{2n}\). Moreover, since \(2n\geq 4\), the higher-dimensional
case of \cite[Theorem~2]{Fried1980} implies that the developing map is a
covering map
\[
        D:\widetilde M\longrightarrow \mathbb R^{2n}\setminus\{p\}.
\]
Viewing \(\mathbb R^{2n}\) as \(\mathbb C^n\), we translate the common fixed
point to the origin. This translation is holomorphic and isometric, so we may
assume \(p=0\). Since \(n\geq 2\), the punctured space
\(\mathbb C^n\setminus\{0\}\) is simply connected. Hence the covering map
\(D\) is a diffeomorphism
\[
        D:\widetilde M\longrightarrow \mathbb C^n\setminus\{0\}.
\]
Because \(D\) is holomorphic and satisfies \(D^*h_{\mathrm{Euc}}=\widetilde h\),
it is a holomorphic isometry. This proves the first assertion.

It remains to describe the deck action in this identification. Since the
holonomy fixes the origin, the affine translation part vanishes:
\(a_\gamma=0\). Thus
\[
        \rho(\gamma)(z)=r_\gamma U_\gamma z.
\]
Using the equivariance \(D\circ\gamma=\rho(\gamma)\circ D\) and identifying
\(\widetilde M\) with \(\mathbb C^n\setminus\{0\}\) by \(D\), every deck
transformation has the form
\[
        \gamma(z)=r_\gamma U_\gamma z,
        \qquad
        r_\gamma>0,\quad U_\gamma\in U(n).
\]

Finally, suppose that \(r_\gamma=1\) for every deck transformation. Then all
deck transformations act isometrically on \((\widetilde M,\widetilde h)\), so
\(\widetilde h\) descends to a K\"ahler metric \(h_0\) on \(M\). Equivalently,
from the deck invariance of \(\pi^*h=e^f\widetilde h\), we get
\[
        e^{f\circ\gamma}\widetilde h
        =
        \gamma^*(e^f\widetilde h)
        =
        e^f\widetilde h,
\]
and hence \(f\circ\gamma=f\) for every \(\gamma\in\mathrm{Deck}(\pi)\). Thus \(f\)
descends to a smooth function \(u\) on \(M\), and
\[
        h=e^u h_0 .
\]
Therefore the Lee form is \(du\), contradicting the assumption that it is not
exact. Hence \(r_\gamma\neq 1\) for at least one deck transformation.
\end{proof}

\begin{proposition}\label{prop:strict-exclusion}
Let \((M^n,h)\) be a connected compact LCK manifold of complex dimension
\(n\geq2\). If \(H_h^C\equiv c\), then the Lee form of \(h\) is exact.
\end{proposition}

\begin{proof}
Assume, to the contrary, that the Lee form is not exact. By
Corollary~\ref{cor:Bochner-flat-cover}, \(h\) is Bochner-flat in the sense of
Tricerri--Vanhecke. Proposition~\ref{prop:Kamishima} therefore identifies the
universal K\"ahler cover with
\[
\widetilde M=\C^n\setminus\{0\},
\qquad
\widetilde h=h_{\Euc},
\qquad
\pi^*h=e^fh_{\Euc}.
\]
Since \(h_{\Euc}\) is flat, \eqref{eq:main-curv-eq} reduces to
\[
0
=
\frac{ce^f}{2}G_{h_{\Euc}}
+
\frac14L_{h_{\Euc}}(A),
\qquad
A_{i\bar j}=f_{i\bar j}.
\]
Lemma~\ref{lem:L-linear-algebra} gives
\begin{equation}\label{eq:flat-Hessian}
f_{i\bar j}=-ce^f\delta_{ij}.
\end{equation}

Choose a deck transformation \(\gamma(z)=rUz\) with \(r\neq1\). Replacing
\(\gamma\) by its inverse if necessary, assume \(0<r<1\). Since
\(e^fh_{\Euc}=\pi^*h\) is deck invariant and
\(\gamma^*h_{\Euc}=r^2h_{\Euc}\), one has
\begin{equation}\label{eq:f-automorphy-flat}
f(rUz)=f(z)-2\log r.
\end{equation}

Suppose first that \(c\neq0\). Differentiating
\eqref{eq:flat-Hessian} in the \(z^k\)-direction gives
\(
f_{i\bar j k}=-ce^ff_k\delta_{ij}
\). Since ordinary derivatives commute, comparison with the same formula
with \(i\) and \(k\) interchanged yields
\[
f_k\delta_{ij}=f_i\delta_{kj}.
\]
For each \(k\), choose \(i\neq k\) and set \(j=i\). Then \(f_k=0\), so
\(\partial f=0\). Because \(f\) is real-valued and
\(\C^n\setminus\{0\}\) is connected, \(f\) is constant. This contradicts
\eqref{eq:f-automorphy-flat}.

Suppose now that \(c=0\). Equation \eqref{eq:flat-Hessian} gives
\(\partial\bar\partial f=0\), so \(f\) is real pluriharmonic. Since
\(\C^n\setminus\{0\}\) is simply connected, there is a holomorphic function
\(F\) on \(\C^n\setminus\{0\}\) such that \(f=\operatorname{Re}F\).
Hartogs' extension theorem extends \(F\) across the origin to an entire
function on \(\C^n\); see \cite[Theorem~2.3.2]{Hormander1990}.

By \eqref{eq:f-automorphy-flat},
\[
        \operatorname{Re}\bigl(F(rUz)-F(z)\bigr)
        =
        -2\log r
\]
on \(\mathbb C^n\setminus\{0\}\). Hence the holomorphic function
\(F(rUz)-F(z)\) has constant real part, and is therefore constant. Thus
there exists \(a\in\mathbb C\) such that
\begin{equation}\label{eq:F-cocycle}
        F(rUz)-F(z)=a,
        \qquad
        \operatorname{Re}a=-2\log r .
\end{equation}
After Hartogs' extension, \(F\) is entire on \(\mathbb C^n\). Therefore both
sides of \eqref{eq:F-cocycle} are entire functions of \(z\), and the identity
extends across the origin. Setting \(z=0\) in \eqref{eq:F-cocycle} gives
\[
        a=F(0)-F(0)=0 .
\]
This contradicts \(\operatorname{Re}a=-2\log r\neq0\), because \(0<r<1\).

Both cases are impossible. Therefore the Lee form is exact.
\end{proof}

\section{Proof of the main theorem}

\begin{proof}[Proof of Theorem~\ref{thm:main}]
Let \((M^n,h)\) be as in the statement. By
Proposition~\ref{prop:strict-exclusion}, the Lee form is exact, so \(h\) is
globally conformally K\"ahler. Proposition~\ref{prop:GCK-branch} then shows
that \(h\) is K\"ahler.

The Chern and Levi-Civita connections of \(h\) now coincide. Since an
algebraic K\"ahler curvature tensor is determined by its holomorphic
sectional curvatures, the K\"ahler polarization identity gives
\[
R^h_{i\bar j k\bar\ell}
=
\frac{c}{2}
\bigl(
 h_{i\bar j}h_{k\bar\ell}
 +h_{i\bar\ell}h_{k\bar j}
\bigr).
\]
See \cite[Chapter~IX, \S7, Proposition~7.6]
{KobayashiNomizuII}. Thus \(h\) has constant holomorphic sectional curvature
\(c\). Since \(M\) is compact, \(h\) is complete, and its simply connected
complete model is, after the appropriate normalization, complex projective
space, complex Euclidean space, or complex hyperbolic space according as
\(c>0\), \(c=0\), or \(c<0\). In particular, if \(c=0\), then
\(R^h\equiv0\), so \(h\) is K\"ahler flat.
\end{proof}

\bibliographystyle{amsalpha}
\bibliography{HSC_revised}

\end{document}